\documentclass[11pt, a4paper]{article}

\usepackage[utf8]{inputenc}
\usepackage[T1]{fontenc}
\usepackage[french, english]{babel}
\usepackage[top=3.5cm, bottom=3cm, left=3.5cm, right=3.5cm]{geometry}
\usepackage{setspace} 

\usepackage{titlesec} 
\usepackage{titleps}
\usepackage[table]{xcolor}
\usepackage{makecell, array, cellspace}
\usepackage{framed} 
\usepackage[shortlabels]{enumitem}

\usepackage[french, refpage, intoc]{nomencl}
\makenomenclature  

\usepackage[scaled=.90]{helvet}
\DeclareRobustCommand{\myBullet}{\scalebox{1.5}{$\cdot$}}

\titleformat{\section}
{\sffamily\bfseries\Large}
{\thesection\hspace*{0.5em}\myBullet}
{0.5em}{}
  
\titleformat{\subsection}
{\sffamily\large\bfseries\color{black!75}}
{\thesubsection\hspace*{0.5em}\myBullet}
{0.5em}{}

\titleformat{\subsubsection}
{\sffamily\bfseries}
{\thesubsubsection}
{1em}{}

\renewcommand{\footnoterule}{\vfill\kern 40pt {\color{black!45}\hrule width 1.0\columnwidth height 0.25pt} \kern 2.6pt}

\newcommand\footnoteWithNoNumber[1]{%
  \begingroup
  \renewcommand\thefootnote{}\footnote{#1}%
  \addtocounter{footnote}{-1}%
  \endgroup
}

\usepackage[hang, flushmargin]{footmisc}

\usepackage{parskip}
\newcommand*\circledSmall[1]{%
	\tikz[baseline=(char.base)]{%
		\node[shape=circle,fill=black!80,draw,inner sep=0.5pt] (char) {\hspace*{-0.5ex}\scriptsize #1};}%
}

\usepackage{amssymb}
\usepackage{amsthm} 
\usepackage{mathtools}
\usepackage{tikz-cd} 

\numberwithin{equation}{section}

\makeatletter
\renewcommand*{\eqref}[1]{%
  \hyperref[{#1}]{\textup{\tagform@{\ref*{#1}}}}%
}
\makeatother

\allowdisplaybreaks 

\newcommand{\tq}{\ensuremath{\; : \;}}

\DeclareMathOperator{\covol}{covol}
\newcommand{\discr}[1]{\ensuremath\deg\left(\Delta_{\min}(#1)\right)}
\DeclareMathOperator{\et}{\text{ét}}
\newcommand{\field}{\ensuremath\F_{3^{2nj}}}
\newcommand{\fieldJ}{\ensuremath\F_{q^{2j}}}
\DeclareMathOperator{\Gal}{Gal}
\DeclareMathOperator{\im}{Im}
\newcommand*{\into}{\hookrightarrow}
\newcommand*{\longinto}[1][\quad]{\;\lhook\joinrel\xrightarrow{#1}\;}

\newcommand*{\minus}{\smallsetminus}

\DeclareMathOperator*{\ord}{ord} 
\DeclareMathOperator{\Reg}{Reg}
\DeclareMathOperator{\rk}{rk}
\DeclareMathOperator{\sep}{sep}
\DeclareMathOperator{\tors}{tors}
\DeclareMathOperator{\vol}{vol}

\newcommand*{\A}{\mathbb{A}} 
\newcommand*{\C}{\mathbb{C}}
\newcommand*{\F}{\mathbb{F}}
\newcommand*{\Pl}{\mathbb{P}} 
\newcommand*{\Q}{\mathbb{Q}}
\newcommand*{\R}{\mathbb{R}}
\newcommand*{\Z}{\mathbb{Z}}

    \DeclareFontFamily{U}{wncy}{}
    \DeclareFontShape{U}{wncy}{m}{n}{<->wncyr10}{}
    \DeclareSymbolFont{mcy}{U}{wncy}{m}{n}
    \DeclareMathSymbol{\Sha}{\mathord}{mcy}{"58} 

\newcommand{\fonction}[5]{
\begin{array}{lrcl}
#1: & #2 & \longrightarrow & #3 \\
    & #4 & \longmapsto & #5
\end{array}
}

\usepackage[linktoc=all, pageanchor, backref=page, colorlinks=true]{hyperref}
\hypersetup{
    linkcolor=[RGB]{29, 90, 122},
    filecolor=[RGB]{0,120,70},
    urlcolor=[RGB]{252, 130, 43},
    citecolor=[RGB]{0, 160, 0},
    bookmarksnumbered
}

\renewcommand*{\backref}[1]{}
\renewcommand*{\backrefalt}[4]{%
    \ifcase #1 (Not cited.)%
    \or        \hspace*{\fill} 	\textcolor{gray}{($\uparrow$~#2.)}
    \else     \hspace*{\fill}  \textcolor{gray}{($\uparrow$~#2.)}
    \fi}

\usepackage[nameinlink,noabbrev]{cleveref}
\crefname{subsection}{subsection}{subsections} 
\crefname{df}{definition}{definitions}
\crefname{thm}{theorem}{theorems}
\crefname{lemme}{lemma}{lemmas}
\crefname{conj}{conjecture}{conjectures}

\renewcommand{\O}{\mathcal{O}}

\newtheoremstyle{StyleThm1}
  {10pt} 
  {10pt} 
  {} 
  {} 
  {\sffamily\bfseries\color{black}}
  {.} 
  {.5em} 
  {} 

\newtheoremstyle{StyleThm2}
  {10pt} 
  {10pt} 
  {\slshape} 
  {} 
  {\sffamily\bfseries\color{black}} 
  {.} 
  {.5em} 
  {} 

\expandafter\let\expandafter\oldproof\csname\string\proof\endcsname
\let\oldendproof\endproof
\renewenvironment{proof}[1][\proofname]{%
	\oldproof[\sffamily\selectfont\upshape\bfseries\color{black} #1 ---\nopunct]%
}{\oldendproof}

\theoremstyle{StyleThm1}
\newtheorem{df}{Definition}[section]
 
\newtheorem{rmqe}[df]{Remark}
\newtheorem{conj}[df]{Conjecture} 

\theoremstyle{StyleThm2}
\newtheorem{thm}[df]{Theorem}
\newtheorem{prop}[df]{Proposition}
\newtheorem{lemme}[df]{Lemma}

\newtheorem{thmLetter}{Theorem}

\newtheorem{corLetter}{Corollary}

\newenvironment{rmq}
  {\pushQED{\qed}\rmqe}
  {\popQED\endrmqe}

\renewenvironment{leftbar}[1][\hsize]
{%
    \MakeFramed{\hsize#1\advance\hsize-\width\FrameRestore}%
}
{\endMakeFramed}

\usepackage{etoolbox}

\crefformat{nomencl}{#2list of symbols#3}

\newif\ifbackrefshowonlyfirst
\backrefshowonlyfirstfalse
\makeatletter
\let\BR@direct@old@hyper@natlinkstart\hyper@natlinkstart
\renewcommand*{\hyper@natlinkstart}{\phantomsection\BR@direct@old@hyper@natlinkstart}
\let\BR@direct@oldBR@citex\BR@citex
\renewcommand*{\BR@citex}{\phantomsection\BR@direct@oldBR@citex}%

\long\def\hyper@page@BR@direct@ref#1#2#3{\hyperlink{#3}{#1}}

\ifx\backrefxxx\hyper@page@backref
    \let\backrefxxx\hyper@page@BR@direct@ref
    \ifbackrefshowonlyfirst
    \fi
\else
    \ifbackrefshowonlyfirst
    \fi
\fi

\RequirePackage{etoolbox}
\patchcmd{\Hy@backout}{Doc-Start}{\@currentHref}{}{\errmessage{I can't seem to patch backref}}
\makeatother

\makeatletter

\def\@@@nomenclature[#1]#2#3#4{%
	{}\hypertarget{#4}
	{}\phantomsection\label{#4}
	\def\@tempa{#2}\def\@tempb{#3}\def\@tempc{#4}%
	\protected@write\@nomenclaturefile{}%
	{\string\nomenclatureentry{#1\nom@verb\@tempa @[{\nom@verb\@tempa}]%
		\begingroup\nom@verb\@tempb\protect\nomeqref{\theequation}%
			|nomlabelref}{\@tempc}}%
		\endgroup
	\@esphack
}
\makeatother

\newcommand{\MyTitle}{\sffamily
On the Mordell--Weil lattice of $y^2 = x^3 + b x + t^{3^n + 1}$ in characteristic~$3$}

\newcommand{\TheAuthor}{\sffamily
Gauthier \scshape Leterrier}

\begin{document}

\title{
	\Large \vspace*{-2em}
	\MyTitle
	\vspace*{-0.5em}
}
\author{\TheAuthor}
\date{\vspace*{-1ex}}

\maketitle 
\thispagestyle{empty}
\vspace*{-1.5em}

\begin{center}
 \parbox{0.8\linewidth}{
 \parskip=0.5em
\small
{\bfseries \sffamily Abstract --} We study the elliptic curves given by $y^2 = x^3 + b x +t^{3^n+1}$ over global function fields of characteristic $3$~; in particular we perform an explicit computation of the $L$-function by relating it to the zeta function of a certain superelliptic curve $u^3 + b u = v^{3^n + 1}$. 

In this way, using the Néron-Tate height on the Mordell--Weil group, we obtain lattices in dimension $2 \cdot 3^n$ for every $n \geq 1$, which improve on the currently best known sphere packing densities in dimensions 162 (case $n=4$) and 486 (case $n=5$). For $n=3$, the construction has the same  packing density as the best currently known sphere packing in dimension $54$, and for $n=1$ it has the same density as the lattice $E_6$ in dimension~$6$.
\vspace{-1em}
}
\end{center}
\footnoteWithNoNumber{%
Date : 27th March 2022. 

Keywords :
elliptic curves, function fields, $L$-functions, Mordell--Weil group, sphere packings.

MSC 2010 (Math. Subject Classification) : 
11G05, 
11M38, 
11T24, 
11H31. 
%
}

\section{Introduction and main results}\label{sect_Main_results}

Following ideas of N. Elkies \cite{ElkiesMW1,ElkiesMW2,ElkiesMW3} and T. Shioda \cite{Shioda_MW_and_sphere_packings} (from the 1990's), one can use elliptic curves over global function fields to get interesting lattice sphere packings of arbitrarily large rank. This is an opportunity to study their arithmetic, and in particular their $L$-function. Interestingly, for our family of elliptic curves, we can compute the $L$-function very explicitly and deduce the main arithmetic invariants of our curves.

For a given positive integer $n \geq 1$ and for an element $b \in \F_{3^n}^\times$, we consider the elliptic curves given by the (affine) Weierstrass equation :
\nomenclature[1\(E\)]{$E_{n, b}$}{Elliptic curve $y^2 = x^3 + b x + t^{3^n + 1}$ over $\F_{3^n}(t)$}{nomencl:Enb}
\begin{align}\label{eq_def_elliptic_curve_E}
E_{n, b} \; : \;		y^2 = x^3 + b x + t^{3^n + 1}
\end{align}
over $\F_{3^n}(t)$.
One of our main results here is the explicit computation of the $L$-function of $E_{n, b}$ over $\F_{3^{2n}}(t)$ for some choice of the parameter $b$. In particular, we can determine the exact value of the (analytic) rank of those elliptic curves.

\begin{leftbar}
\vspace*{-10pt}
\begin{thmLetter}\label{thm_L_function}
Let $n \geq 1$ be an integer and set $q = 3^n$. Let $b \in \F_{q}^\times$ be any element such that $b^{\tfrac{q-1}{2}} = (-1)^{n+1}$. 
\nomenclature[1\(Q\)]{$q$}{$q=3^n$ where $n \geq 1$}{nomencl:q}%

Then the $L$-function (as defined in \cref{eq_definition_Lfunction}) of the elliptic curve $E_{n, b}$ over $\F_{q^2}(t)$ given by \eqref{eq_def_elliptic_curve_E} is equal to
\begin{equation}
L\left( E_{n, b} / \F_{q^2}(t)	,\;	 T \right)  	=   (1 - q^2 T)^{2 \cdot 3^n}.
\end{equation}
In particular, the analytic rank of $E_{n, b}$ over $\F_{q^2}(t)$ is equal to $2 \cdot 3^n$.
\end{thmLetter}
\vspace*{-15pt}
\end{leftbar}

Let us explain how to construct a lattice from those curves.
In general, if $E$ is an elliptic curve over a global function field $K = k(X)$, where $X$ is a smooth projective algebraic curve over a finite field $k$ (we will mostly focus on the case $X = \Pl^1$), then $E(K)$ is a finitely generated abelian group, by Mordell--Weil theorem (generalized by Lang and Néron ; see \cite[theorem III.6.1]{SilvermanAAEC}). We denote the identity element by $O_E$.

Given a Weierstrass equation for $E$ over $k(X)$, we have a degree-2 cover $x : E  \to \Pl^1$ given by the $x$-coordinate. For every $P \in E(k(X)) \minus \{O_E\}$, we can see $x(P) \in k(X)$ as a rational map $X \dashrightarrow \Pl^1$. We can therefore define the \textsl{naive height} as
\[ 		\fonction{h}{ E(K) }{ \Z_{\geq 0} }
						{ P }{  \begin{cases}
									\mathsf \deg(x(P))		&\text{if } P \neq O_E \\
									0 							&\text{else}
								\end{cases}	} \]
(if $x(P) \in k$ is constant, its degree is set to be $0$).
We define the (canonical) Néron--Tate height as
\nomenclature[1\(H\)]{$\hat{h}$}{Néron--Tate height $\hat{h} : E(K) \to \R_{\geq 0}$}{nomencl:height}
\begin{equation}\label{eq_def_Neron_Tate_height}
\hat{h}(P)		:=		\lim_{n \to + \infty} 4^{-n} h(2^n P) \in \R
\end{equation}
for every $P \in E(K)$. It is a quadratic form, which is positive-definite on $E(K) / E(K)_{\tors}$ (\cite[theorem III.4.3]{SilvermanAAEC}), where $E(K)_{\tors}$ denotes the torsion subgroup of $E(K)$. 
Therefore, we obtain a lattice, called the \textsl{Mordell--Weil lattice} of $E$ over $K$. We introduce a convenient sublattice, namely :

\begin{df}\label{def_narrow_MWL}
\nomenclature[1\(E\)]{$E(K)^0$}{Narrow Mordell--Weil lattice}{nomencl:narrowMWL}
The \textsl{narrow Mordell--Weil lattice} of $E$ over $K$ consists of all the points $P \in E(K)$ such that for every place $v$ of $K$, the reduction $\overline{P}$  is a non-singular point on the reduction $\overline{E_v}$ of a minimal integral Weierstrass model $E_v$ of $E$ at $v$.
It is denoted by $E(K)^0  \subset E(K)$.
\end{df}

Now, given a lattice $L \into \R^d$, let 
\nomenclature[1\(L\)]{$\lambda_1(L)$}{Length of one of the shortest non-zero vectors of a lattice $L$}{nomencl:min-norm-L}
\begin{equation}\label{eq_def_lambda1}
\lambda_1(L) 		:=		 \min\left\{ \| v \|  \tq v \in L \minus \{0\} \right\}
\end{equation}
be the length of one of its shortest non-zero vectors.
Then the translates $B+L$ of the euclidean ball $B = B\left( 0, \frac{\lambda_1(L)}{2} \right) \subset \R^d$ by points of $L$ defines a \textsl{lattice packing} of balls. Its \textsl{density} is defined as the proportion of the space covered by $B+L$, i.e.,
\nomenclature[1\(D\)]{$D(L)$}{Packing density of a lattice $L$ : $D(L) \in [0,1]$}{nomencl:packingDensity}
\begin{equation}\label{eq_def_packing_density}
D(L) := \limsup\limits_{r \to +\infty} \dfrac{ \vol( (B+L) \cap B(0, r)) }{ \vol(B(0,r)) } \in [0, 1].
\end{equation}
In the case of such a lattice packing, we can simplify the expression into\linebreak
$D(L) = \dfrac{ (\lambda_1(L) / 2)^d \cdot \vol(B(0,1)) }{ \vol(\R^d / L) }$. This motivates us to consider the following normalization :%
\vspace{-1em}%
\nomenclature[1\(D\)]{$\delta(L)$}{Center density of a lattice packing $L$}{nomencl:CenterDensity}%
\begin{df}\label{def_center_density}
The \textsl{center density} of a packing of balls given by a lattice $L \into \R^d$~as
\[
\delta(L) := \dfrac{ (\lambda_1(L) / 2)^d }{ \vol(\R^d / L) }.
\qedhere	\]
\end{df}

The maximal sphere packing density $\displaystyle P_d := \sup_{\substack{L \subset \R^d \text{ lattice}}} D(L)$, in a given dimension $d$, is known exactly only if $d \leq 8$ or if $d=24$. Minkowski used a non-constructive argument to show that $P_d \geq 2 \cdot 2^{-d}$ (we refer to \cite{Conway_Sloane_book} for more details). Very little seems to be known about explicit lattice constructions that reach this lower bound, let alone exceed it, if the dimension is large enough. 

It turns out that as a corollary of \cref{thm_L_function} and of results of Shioda, we get a lower bound on the sphere packing density of the narrow Mordell--Weil lattice of the elliptic curves $E_{n, b}$. 
\begin{leftbar}
\vspace*{-10pt}
\begin{corLetter}\label{cor_packing_density_MWL}
Let $n \geq 1$ be an integer, fix $b \in \F_{3^n}^\times$ as in \cref{thm_L_function}, and set $q = 3^n$.

\nomenclature[1\(L\)]{$L_n$}{Narrow Mordell--Weil lattice $L_n := E_{n, b}\left( \F_{q^2}(t) \right)^0$}{nomencl:NarrowMWLofEnb}
Let $L_n := E_{n, b}\left( \F_{q^2}(t) \right)^0$ be the narrow Mordell--Weil lattice of the elliptic curve $E_{n, b}$ over $\F_{q^2}(t)$, as defined in \cref{def_narrow_MWL}.

Then the rank of $L_n$ is $2 \cdot 3^n$ and its center density satisfies the lower bound
\begin{align}\label{eq_cor_packing_density_MWL}
\delta(		L_n		) 		\geq
\left( \dfrac{3^{n-1} + 1}{4} \right)^{3^n} 	
\cdot 		
3^{- n \left( \tfrac{3^{n-1} - 1}{2} \right)  		-		\tfrac{1}{2}}
\end{align}
\end{corLetter}
\vspace*{-25pt}
\end{leftbar}

In particular, for $n \in \{1, ..., 7\}$, we get the following values, gathered in the \hbox{table} below.
\vspace{0.5em}
\begin{center}
\newcolumntype{C}[1]{%
 >{\vbox to 3.5ex\bgroup\vfill\centering\arraybackslash}%
 p{#1}%
 <{\vskip-\baselineskip\vfill\egroup}}  
\begin{tabular}{>{} C{0.5cm} <{} C{1.25cm} <{} C{4.5cm} <{} C{4cm} <{}}  
\hline
 \rowcolor{black!5} $n$ & rank of $L_n$ & $\log_2\left( \delta(L_n) \right) \geq$ & \makecell{Best \emph{lattice}  packing \\ density known so far} \\
   \hline
   	$1$ & $6$ & $\log_2\left(\sqrt{3} / 24 \right) \simeq -3.79248$ 	&
   	\makecell{$\delta(E_6) = \frac{ \sqrt{3} }{ 24 }$ \\ (\footnotesize \cite{Conway_Sloane_book}, p. xix)} \\
  \hline	
	$2$ & $18$ & {\color{gray} $\log_2\left(\tfrac{\sqrt{3}}{27} \right) \simeq -3.962406$} &  \shortstack{$-3.79248$ \\ \cite{Conway_Sloane_book}, p. xix }\\
  \hline	  
  $3$ & $54$ & $\log_2\left(\tfrac{\sqrt{3} \cdot 5^{27}}{2^{27} \cdot 3^{13}} \right) \simeq 15.88002$		& 
  \shortstack{15.88 \\ \footnotesize{ (Elkies \cite{Conway_Sloane_book}, p. xviii) }} \\ 
  \hline
  $4$ & $162$ & 144.1852	& \shortstack{130.679 \\ \cite{Refinement_Craig}} \\
  \hline
  	$5$ & $486$ & 741.1001	&  \shortstack{703.05 \\ \cite{Ball_lower_bound_on_packing_density} } \\
  \hline
  	$6$ & $1458$ & {\color{gray} 3172.032}	& \shortstack{3236.6 \\ \cite{Ball_lower_bound_on_packing_density} } \\
   \hline
\end{tabular}
\end{center}
\vspace{2ex}

We see that in dimensions $6$ and $54$, we get the same density as the previous densest known lattice packings of balls (in fact no construction is provided for the 54-dimensional lattice $MW_{54}$ listed in \cite{Conway_Sloane_book}, p. xx). 
Moreover, in dimensions 162 and 486, we improve the current records. But in dimension 18, another construction achieves a higher packing density, and in dimensions above 1458, non-constructive lower bounds are the best known so far.

\subsection{Outline of the proofs}

\Cref{thm_L_function} is proved in \cref{sect_pf_thm_A} by performing an explicit computation of the $L$-function. This requires counting the number of points on the reduction of $E_{n, b}$ modulo all the places of $\F_{q^2}(t)$, which involves sums of Legendre symbols, introduced in \cref{subsec_definition_Legendre_sums}.
 
Those sums can be determined thanks to an auxiliary superelliptic curve over $\F_{q}$ (see \cref{subsec_superelliptic_curve}), and using the fact that $x \mapsto x^3 + bx$ is an additive map in characteristic $3$ (see \cref{lemma_linear_algebra,lemma_sums_sigmab}). 
Finally, the number of points over $\F_{q^2}$ of this auxiliary superelliptic curve can be computed essentially because its jacobian is isogenous to a power of a \textit{supersingular} elliptic curve. 

The idea behind this approach was inspired by the work of N. Elkies \cite{ElkiesMW1}, where a counting argument about hyperelliptic curves has been used. In our case, this will get replaced by a \emph{superelliptic} curve (see \cref{subsec_superelliptic_curve}).

In both works, the elliptic curves (over function fields of characteristic 2 and 3 respectively) are isotrivial -- we also say equivalently "potentially constant". But in our case $E_{n, b}$ is a \textit{cubic} twist of a constant curve (i.e., defined over $\F_{q}$ ; see \cref{prop_E_isotrivial}), while the elliptic curves studied by Elkies were \textit{quadratic} twists of a constant curve, which is what allowed to compute of the rank and the $L$-function.

Finally, \cref{cor_packing_density_MWL}, proved in \cref{sect_pf_cor_A}, follows from the use of Birch--Swinnerton-Dyer formula (known in this case, because $E_{n, b}$ is \textit{isotrivial}, see \cref{thm_BSD_true_for_isotrivial} and \cref{prop_E_isotrivial}), as well as a result of Shioda on the lower bound of the height of points in the narrow Mordell-Weil lattice (\cref{thm_Lower_bound_on_minimal_norm}).
In \cref{subsec_sharpness}, we discuss the sharpness of inequality \eqref{eq_cor_packing_density_MWL}.

For the convenience of the reader, some frequently used notations are gathered in a \cref{SectNomenclature} at the end of this document. 

\begin{rmq}\label{rmk_Shioda_algo}
\newcommand*{\rmE}{\mathrm{E}}
Before continuing, we mention that Shioda's results in \cite{Shioda_algorithm_Picard_nb} (especially remark 10 therein) tell us that if $m$ is an even integer and $3^e \equiv -1 \pmod{2m}$ for some integer $e \geq 1$, taken to be minimal, then the rank of $\rmE_m(  \overline{\F_3}(t) )$ is $f(\rmE_m) - 4$, where $f(\rmE_m)$ denotes the conductor of the elliptic curve $\rmE_m : y^2 = x^3 + x + t^m$ over $\F_3(t)$.

We are going to investigate the situation where $m = 3^n + 1$ for some integer $n \geq 1$, which is \emph{not} directly covered by the above result. 
Even if it did apply (e.g., via theorem 1, \emph{ibid.}), we would need to know over what finite field of constants the rank is achieved, so anyway we need to use another technique to determine the rank. 

It is also possible to express the L-function of $\rm{E}_m$ over $\F_p(t)$ for any odd prime $p$ explicitly in terms of Jacobi sums, which allows to get another proof of \cref{thm_L_function}. See also \cref{rmk_Shioda_algo_2} below.
\end{rmq}

\pagebreak

{\hypersetup{linkcolor=black!75}
\section[Proof of theorem A]{Proof of \texorpdfstring{\cref{thm_L_function}}{theorem A} }
\label{sect_pf_thm_A}
}

\subsection{Definition of the \emph{L}-function}

In this paragraph, we consider a finite field $k = \F_{|k|}$, and we set $K = k(t)$.
Recall that the set of places $v$ of $K$ (i.e., an equivalence class of absolute values on $K$, which are necessarily trivial on $k$ and are non-archimedean) 
is in bijection with the set of closed points of $\Pl^1_k$, which is itself in bijection with the set of Galois orbits of $\bar{k}$-rational points in $\Pl^1(\overline{k})$.

We denote by $\infty$ the place of $K$ associated to the point $[1:0] \in \Pl^1_k$ (it is given by $-\deg$).

Let $E$ be an elliptic curve over $K$.
For any place $v$ of $K$, we let $E_v$ be a minimal \textit{integral} Weierstrass model at $v$. We let $\overline{E_v}$ be its reduction modulo an uniformizer $\pi_v$ of the ring of integers $\O_v \subset K_v$ of the completion $K_v$ of $K$ at $v$, that is: $\overline{E_v} := E_v \times_{\O_v} \O_v / (\pi_v)$. This is a projective plane cubic curve (possibly singular) over the finite field $\F_v := \O_v / (\pi_v)$, and its isomorphism class does not depend on the choice of a minimal integral Weierstrass model $E_v$ (this follows from proposition VII.1.3 (b) in \cite{SilvermanAEC}).

We now define the integers
\nomenclature[1\(A\)]{$A_E(v, j)$}{Given an elliptic curve $E$, a place $v$ and a multiple $j $ of $\deg(v)$, we set $A_E(v, j)  :=\linebreak "|k"|^j		+ 1	-	"| \overline{E_v}(\F_{"|k"|^j}) "|$}{nomencl:AEvj}
\begin{equation}\label{eq_definition_sums_A}
\begin{split}
A_E(v, j)  &:= 		 |k|^j		+ 1	-	\left| \overline{E_v}(\F_{|k|^j}) \right| 				,	\\
a_v(E) 		&:=		 A_E(v, \deg(v)) = | \F_v | + 1	- 	\left| \overline{E_v}(\F_v) \right|	 ,
\end{split}
\end{equation}
where $j \geq 1$ is any integer multiple of $\deg(v) := [\F_v : k]$ (so in particular $\F_{|k|^j}$ is an extension of $\F_v$).
Notice that $a_v(E)$ is equal $0$ if $E$ has additive reduction at $v$, and $\pm 1$ if $E$ has multiplicative reduction at $v$ (this follows from proposition III.2.5 in \cite{SilvermanAEC}, see also section 2.10 in \cite{Washington_EllCurves}).

We define the \textit{local factor} at $v$ as
$$L_v(E/K, T) := \begin{cases}
1 - a_v(E)	 T^{\deg(v)} + |k|^{\deg(v)} T^{2 \deg(v)} 		&\text{if $E$ has good reduction at $v$} 
 \\
1 - a_v(E)  T^{\deg(v)}			&\text{else.}
\end{cases}$$
The \textit{L}-function is defined as 
\begin{equation}\label{eq_definition_Lfunction}
L(E/K, T) := \prod_{ \substack{v \text{ place} \\ \text{of } K}} 
					L_v(E/K, T)^{-1} \in \Z[[T]].
\end{equation}

One can re-write the $L$-function as follows (this is Lemme 1.3.15 in \cite{Griffon_PhD}), by an elementary computation, where $[w]$ is the place corresponding to $w$~:
\begin{align}\label{eq_log_Lfunction_as_sum}
\log L(E / K, T)
&=
\sum_{j \geq 1} 		
\Bigg( 		\sum_{		w \in \Pl^1( \F_{|k|^j} )		}  A_E( [w] , j)		 \Bigg)
\dfrac{ T^j }{ j }.
\end{align}

\subsection{Definition of the relevant Legendre sums}\label{subsec_definition_Legendre_sums}

We first analyze the reduction types of the elliptic curve $E_{n, b}$ over $\F_{q^2}(t)$, which we state as a proposition for later use. To this end, we recall some standard notations.

\begin{df}\label{def_Tamagawa_nb_discr_conductor}
Let $k$ be a finite field, and let $X$ be a smooth projective geometrically irreducible algebraic curve over $k$. Denote by $g_X$ its genus. Set $K = k(X)$ and let $E$ be an elliptic curve over $K$.

\nomenclature[1\(D\)]{$\discr{E/K}$}{Degree of the minimal discriminant of an elliptic curve $E/K$}{nomencl:discr}
\nomenclature[1\(F\)]{$f(E/K)$}{Degree of the conductor of an elliptic curve $E/K$}{nomencl:cond}
\nomenclature[1\(C\)]{$c(E/K)$}{Product of the Tamagawa numbers $c_v(E/K)$ of an elliptic curve $E/K$}{nomencl:tamagawa}
\begin{enumerate}
\item We denote by
$\Delta_{\min}(E/K)$ the minimal discriminant of $E/K$ (as in \cite[exercise 3.35]{SilvermanAAEC}). It is a divisor on the curve $X$.

\item 
We denote by $f(E/K)$ the degree of the conductor divisor of $E/K$ (see \cite[exercise 3.36]{SilvermanAAEC}).

\item For each place $v$ of $K$, we denote by $c_v(E/K)$ the local Tamagawa factor of $E/K$ at $v$, i.e., the number of irreducible components of the special fiber of the Néron model of $E$ at $v$ that have multiplicity 1 and are defined over the residue field $\F_v$ at $v$. 
We also set $c(E/K) := \prod\limits_{v \in |X|} c_{v}(E/K)$.
\qedhere
\end{enumerate}
\end{df}

\begin{prop}\label{prop_reduction_types}
Let $E_{n, b}$ be the elliptic curve $y^2 = x^3 + bx + t^{3^n+1}$ over $K_n := \F_{3^{2n}}(t)$ (where $b \in \F_{3^n}^\times$ and $n \geq 1$ are fixed).

Then $E_{n, b}$ has good reduction at all places $v \neq \infty$ and has bad additive reduction of type IV at $v = \infty$, with the following invariants :
\begin{align*}
\discr{E_{n,b} / K_n} 	&= 12 \lceil (3^n+1)/6 \rceil = 2 \cdot (3^n + 3),
\\
f(E_{n,b} / K_n) 			&= \discr{E_{n,b} / K_n}  - 2,
\\
c(E_{n,b} / K_n) &= c_{\infty}(E_{n,b} / K_n) = 3.
\end{align*}
\vspace*{-1cm}
\end{prop}
\begin{proof}
The curve $E_{n, b}$ has Weierstrass discriminant $-b^3 \in \F_{3^n}^\times$ according to proposition A.1.1.(b) in \cite{SilvermanAEC}. In particular,  $E_{n, b}$ had good reduction at all places $v \neq \infty$ and $y^2 = x^3 + bx + t^{3^n+1}$ is a minimal integral Weierstrass model at all $v \neq \infty$.

We follow Tate's algorithm as written down in \cite[IV.9, p. 366]{SilvermanAAEC}.
Let $m := 3^n + 1$ and $\mu := \lceil \frac{m}{6} \rceil \geq 1$. It is easy to see that $m \equiv 4 \pmod 6$ (e.g., by induction on $n$), so $6\mu - m = 2$.

The affine equation $E_{\infty} : y^2 = x^3 + b x t^{-4 \mu} + t^{m - 6 \mu}$ is a minimal integral Weierstrass model at $v = \infty$, where we take $\pi_v := t^{-1}$ as uniformizer.

We have the following coefficients, as defined in \cite[IV.9, p. 364]{SilvermanAAEC} :
\[
b_2 = 0,	\qquad
b_4 = 2b t^{-4\mu},	\qquad
b_6 = 4t^{m - 6 \mu}, 	\qquad
b_8 = - \frac{1}{4} \cdot (2bt^{-4\mu} ) ^2.
\]
The singular point on the reduction of $E_{\infty}$ modulo $\pi$ is $(\overline{0}, \overline{0})$, which means that the condition in \emph{Step 2} of Tate's algorithm (as in \cite[IV.9, p. 366]{SilvermanAAEC}) is satisfied.

Since $6\mu - m = 2$, the constant coefficient $a_6 := t^{m - 6 \mu}$ is equal to (hence divisible by) $\pi^2$. Moreover, $b_8$ is divisible by $\pi^3$, but $b_6$ is not divisible by $\pi^3$.
Therefore, Tate's algorithm stops at \emph{Step 5}, which states that $E$ has bad additive reduction of Kodaira--Néron type IV.
From there, we know that the Tamagawa number at $v=\infty$ is $c_{\infty}(E) = 3$, 
and that the local conductor is $f_v = v(\Delta) - 2$ and $v_{\pi}(\Delta) = 12 \mu = 12 \lceil m/6 \rceil$, since the Weierstrass discriminant is $\Delta = -8b_4^3 - 27b_6^2 			=		 -8 \cdot 8b^3 \pi^{12 \mu}$. 
\end{proof}

\nomenclature[1\(L\)]{$\lambda_{k}$}{Legendre symbol $k^\times \to \{ \pm 1\}$ of a finite field $k$}{nomencl:Legendre}
When $k$ is a finite field of odd cardinality, let 
\begin{align*}
\lambda_k : k^{\times} \to \{ \pm 1\} \into \C^{\times},
\qquad 
x \mapsto x^{\frac{|k| - 1}{2}}
\end{align*}
be the Legendre symbol which is the unique character of order $2$ on $k^{\times}$.

\begin{rmq}\label{rmk_Legendre}
It is important to be careful about the subscript $k$ in $\lambda_{k}$, because when $q'$ is a power of some odd prime power $q$, the restriction of $\lambda_{\F_{q'}}$ to the subfield $\F_q$ is \emph{not} equal to $\lambda_{\F_{q}}$ (e.g., $\lambda_{\F_{q^2}}(x) = 1$ for every $x \in \F_q$).
\end{rmq}

According to \cref{eq_log_Lfunction_as_sum}, computing the $L$-function of $E_{n, b}$ amounts to determining the sums
\nomenclature[1\(S\)]{$S_b(n, j)$}
{Sum of $A_{E_{n, b}}(w, j)$ over $w \in \Pl^1( 		\F_{  (q^2)^j }		)$}{nomencl:Sbnj}
\begin{equation}\label{eq_definition_sums_Sb}
S_b(n, j) := 		\sum_{		w \in \Pl^1( \F_{  (q^2)^j } )		}  A_{E_{n, b}}(w, j),
\end{equation}
where $q = 3^n$, as we have $\displaystyle \log L(E_{n, b} / \F_{q^2}(t), T) = \sum\limits_{j \geq 1} 	S_b(n, j) \frac{ T^j }{ j }$.
%
%
From the definition \eqref{eq_definition_sums_A}  of  $A_E(w, j)$ and of $E_{n, b}$, we see that the above sum is equal to
\begin{equation}\label{eq_sum_Sb_as_Legendre}
S_b(n, j) = 		- \sum_{			w, 	x	 \in \F_{  (q^2)^j }			}  		
\lambda_{ 	\F_{  (q^2)^j }	 } 
\left( x^3 + b x + w^{3^n + 1} \right).
\end{equation}
Notice that we can discard the terms with $w = [1:0]$ since we have $A_{E_{n, b}}(\infty, j) = 0$ for every $j \geq 1$ by \cref{prop_reduction_types}.

The strategy to evaluate those sums $S_b(n, j)$ consists of two steps :  
\begin{enumerate}[label=\protect\circledSmall{ \color{white} \arabic*}]
\item \label{item_step_superelliptic}
First, we will compute the number of points on a certain superelliptic curve $C_{n, b}$, given by $v^{3^n+1} = u^3 + bu$ over $\field$ for every $j \geq 1$, where $b \in \F_{3^n}^\times$ is chosen as in \cref{thm_L_function}.

\item \label{item_step_sums_sigmab}
Secondly, we study the sums
\nomenclature[1\(S\)]{$\sigma_b(j, t)$}
{Sum of $\lambda_{ \F_{3^{2 nj}} }(x^3 + b x + t)$ over $x \in \F_{ 3^{2 nj } }$}{nomencl:sigmabjt}
\begin{equation}\label{eq_definition_sums_sigmab}
\sigma_b(j, t)	:=	\sum_{  x \in \F_{ 3^{2 nj } }  }  \lambda_{ \F_{3^{2 nj}} }(x^3 + b x + t),
\end{equation}
where $t \in \F_{3^{2 nj}}$ and $j \geq 1$ is any integer.
\end{enumerate}

\mathversion{bold}
\subsection{Number of points on the superelliptic curve $C_{n, b}$}\label{subsec_superelliptic_curve}
\mathversion{normal}

For any $n \geq 1$, let $C_{n, b}^{\rm aff}$ be the affine curve $v^{3^n + 1} = u^3 + b u$, defined over $\F_{3^n}$, where $b \in \F_{3^n}^\times$ satisfies $b^{(3^n - 1)/2} = (-1)^{n+1}$ as in the statement of \cref{thm_L_function}. Note that $C_{n, b}^{\rm aff}$ is smooth.
\nomenclature[1\(C\)]{$C_{n, b}$}{Superelliptic curve with affine model $v^{3^n + 1} = u^3 + b u$ over $\F_{3^n}$}{nomencl:superelliptic}

There is a smooth projective irreducible curve $C_{n, b}$ over $\F_{3^n}$ (unique up to isomorphism) 
such that its function field is the same as the one of $C_{n, b}^{\rm aff}$. We say that $C_{n, b}$ is a \emph{superelliptic curve}.

It turns out that $C_{n,b}$ has a unique point at infinity, defined over $\F_{3^n}$ (see proposition 2 in \cite{Galbraith_Paulus_Smart_Superelliptic}), so that
$|C_{n, b}(k)| 		=		 \big| C_{n, b}^{\rm aff}(k) \big| + 1$ for every finite extension $k$ of $\F_{3^n}$.

The key point is that we will be able to deduce the number of points $|C_{n, b}(  \field )| $, for all $j \geq 1$, just from the computation of $|C_{n, b}(  \F_{ 3^{2n} }  )|$. %
Now, we can compute $|C_{n, b}(  \F_{ 3^{2n} }  )|$ because the norm map
$$\F_{3^{2n}}^{\times}  \longrightarrow \F_{3^{n}}^{\times},   
\qquad
v \longmapsto v^{3^n} \cdot v = v^{3^n + 1} =: w$$
is a surjective morphism, with kernel of size $\dfrac{3^{2n} - 1}{3^n - 1} = 3^n  +  1$. 

Therefore, we get
\begin{equation}\label{eq_superelliptic_curve_card1}
|C_{n, b}(  \F_{ 3^{2n} }  )|	
= 1 + 3 + (3^n + 1) \sum_{w \in		 \F_{3^{n}}^{\times}	 }  \# \{ u \in \F_{3^{2n}}		\tq		u^3 + b u = w  \}
\end{equation}

We determine each term in the latter sum in the following lemma (applied to the case where $p := 3$).
\begin{lemme}
\label{lemma_linear_algebra}
Let $p$ be an odd prime, $n \geq 1$ be an integer, set $q = p^n$ and let 
$b \in \F_{p^n}^\times$ be any element such that 
\begin{equation}\label{eq_hypothesis_on_b}
\mathrm{Nr}_{\F_q / \F_p}(b)			\; = \;			b^{	\tfrac{p^n - 1}{p - 1}	} 	= (-1)^{n+1}.
\end{equation}
 Then we have 
\[  \#\{	x \in \F_{q^2}	\tq		x^p + b x \in \F_q	\} =	
p^{n+1} = p \cdot q.		\qedhere 		\] 
\end{lemme} 
\begin{proof}
Consider the maps $f, g_b : \F_{q^2} \to  \F_{q^2}$ defined by $f : x \mapsto x^q - x$ and $g_b : x \mapsto x^p + b x$.
The key point is that these maps are endomorphisms of the additive group $(\F_{q^2}, +)$ seen as vector space over $\F_p$, and we can describe the set
$\{	x \in \F_{q^2}	\tq		x^p + b x \in \F_q	\}$ as the kernel of $f \circ g_b$. Thereby, the proof essentially boils down to a basic argument of linear algebra. A direct computation shows that $f \circ g_b = g_b \circ f$ (using the fact that $b \in \F_q^\times$).

The rank-nullity theorem yields
\begin{equation}\label{eq_dimension_rank_thm}
\dim( \ker(g_b \circ f) ) 	=		\dim(\ker(f))		+		\dim( \ker(g_b) \cap \im(f) ),
\end{equation}
where the dimensions are taken over $\F_p$.

It is clear that $\dim(\ker(f)) = n$, since $q = p^n$, and that $\ker(g_b)$ has dimension $1$ since it consists of roots in $\overline{\F_p}$ of the separable polynomial $X^p + b X$ which has degree $p$, and all those roots actually lie in $\F_{q^2}$. Indeed, if $x^p = -bx$ then
\begin{align}
\begin{split}
x^{p^{n}} &= \hspace*{0.7ex}
(-b)^{ 1 + p + \cdots + p^{n-1}} \cdot x
= (-b)^{ \tfrac{p^n - 1}{p - 1} } \cdot x 		\overset{ \eqref{eq_hypothesis_on_b} }{=}
 (-1)^{ \tfrac{p^n - 1}{p - 1} } \cdot  (-1)^{n+1} x 
\\&
\hspace*{-0.7ex}\overset{p \text{ odd}}{=} (-1)^n  \cdot  (-1)^{n+1} x 			\;=\;		 -x,
\end{split}\label{eq_roots_belong_to_Fq}
\end{align}
which implies that $x^{q^2} = (x^q)^q = (-x)^q =x$, i.e., $x \in \F_{q^2}$ as claimed.

The above computation \eqref{eq_roots_belong_to_Fq} also shows that any element $x \in \ker(g_b)$ satisfies $x^{p^n} = -x$, so that $f(x) = -2x$, which shows that $x \in \im(f)$ (recall that $p$ is odd, so $-2 \in \F_p^{\times}$ is invertible). In other words, we have $\ker(g_b) \cap \im(f) = \ker(g_b)$. Finally we get $\dim( \ker(f \circ g_b)) = \dim( \ker(g_b \circ f))  = n + 1$ from \cref{eq_dimension_rank_thm}, which yields 
$$\#\{	x \in \F_{q^2}	\tq		x^p + b x \in \F_q	\} = | \ker(f \circ g_b) | = p^{n+1},$$ which is what we wanted to prove.
\end{proof}

Therefore, from \cref{eq_superelliptic_curve_card1} and the above \cref{lemma_linear_algebra} (applied to $p = 3$), we get
\begin{align*}
|C_{n, b}(  \F_{ 3^{2n} }  )|	
&= 1 + 3 + (3^n + 1)      	\left( \# \{ u \in \F_{3^{2n}}		\tq		u^3 + b u \in \F_{3^n}  \}   -   3 \right)
\\&=
1 + 3^n \cdot 3^{n+1}
\end{align*}

We now consider $C_{n, b}$ as a curve over $\F_{3^{2n}}$ (instead of a curve over $\F_{3^n}$).
Let us write $\omega_k$ for the eigenvalues of the Frobenius endomorphism of $x \mapsto x^{3^{2n}}$ acting on $H^1_{\et}(C_{n, b} \times \overline{\F_3}, \Q_{\ell})$, where $1 \leq k \leq 2 \cdot g(C_{n, b})$ and $g(C_{n, b})$ denotes the genus of $C_{n, b}$ and $\ell \neq 3$ is a prime. 
It is known from the Weil conjectures that the $\omega_k$ are the reciprocal of the roots of the numerator (in $\Z[T]$) of zeta function $Z(C_{n, b} / \F_{3^{2n}}, T)$ ; in particular, they can be seen as complex numbers and their modulus is known to be equal to $|\omega_k| = \sqrt{3^{2n}} = 3^n$.
Thereby, Lefschetz trace formula tells us that
\begin{align*}
|C_{n, b}(  \F_{ 3^{2n} }  )|	
= 3^{2n} + 1 - \sum_{k=1}^{2 g(C_{n, b})} \omega_k.
\end{align*}
The genus of $C_{n, b}$ is equal to $g(C_{n, b}) = 3^n$ (see proposition 2 in \cite{Galbraith_Paulus_Smart_Superelliptic}). 
Hence we get
\begin{align*}
|C_{n, b}(  \F_{ 3^{2n} }  )|	
=
1 + 3^{2n+1} = 3 \cdot 3^{2n} + 1 		=   3^{2n} + 1 - \sum_{k=1}^{2 \cdot 3^n} \omega_k,
\end{align*}
which implies $-2 \cdot 3^{2n} = \sum\limits_{k=1}^{2\cdot 3^n} \omega_k$. Because the $\omega_k \in \C$ satisfy $|\omega_k| = \sqrt{3^{2n}} = 3^n$, this forces $\omega_k = -3^n$ for every $k$ (e.g., by taking the real part of the latter sum). 
%
We conclude that for every $n \geq 1$ and every $j \geq 1$ :
\begin{align*}
|C_{n, b}(  \field  )|	 =
3^{2nj} + 1 - 2 \cdot 3^n \cdot (-3^n)^j.
\end{align*}
This completes the step \ref{item_step_superelliptic} announced above.
We can sum up what we have obtained above in terms of the zeta function of $C_{n, b}$ :
\begin{prop}\label{prop_zeta_function_superelliptic}
Let $n \geq 1$ be an integer and let $b \in \F_{3^n}^\times$ be as in \cref{thm_L_function}.
The zeta function of the superelliptic curve $C_{n, b}$ over $\F_{3^{2n}}$ is given by
\begin{equation*}
Z\big(		C_{n, b} / \F_{3^{2n}}	,\		T		\big)
\; = \;
\dfrac{			(1 + 3^n T)^{2 \cdot 3^n} 			}{ (1 - T)(1 - 3T) }.
\end{equation*}
In particular, for every $j \geq 1$, we have 
\[ |C_{n, b}(  \field  )|	 =
3^{2nj} + 1 - 2 \cdot 3^n \cdot (-3^n)^j. 		\qedhere \]
\end{prop}

\mathversion{bold}
\subsection{Evaluating the sums \texorpdfstring{$\sigma_b(j, t)$}{sigma\textsuperscript b(j, t)} }\label{subsec_evaluation_sums_sigma}
\mathversion{normal}

This paragraph is devoted to the explicit computation of the sums $\sigma_b(j, t)$ defined in \cref{eq_definition_sums_sigmab}. Then we will conclude the proof of \cref{thm_L_function}.
\begin{lemme}\label{lemma_sums_sigmab}
Let $n \geq 1$ be an integer, set $q = 3^n$ and fix $b \in \F_{3^n}$ such that $\lambda_{\F_{3^n}}(b) = (-1)^{n+1}$.
Let $j \geq 1$ be any integer.
Consider the map $g_{b, j} : \F_{q^{2j}} \to  \F_{q^{2j}}$ defined by $g_{b, j} : x \mapsto x^3 + b x$.

Then for every $t \in \fieldJ$ we have :
\begin{align*}
\sigma_b(j, t)	 =
\begin{cases}
-2 \cdot (-3^n)^{j}		&\text{ if } t \in \im(g_{b, j}) \\
(-3^n)^{j}					&\text{ otherwise. }
\end{cases}
\end{align*} %
\end{lemme}

\newlength{\ItemLength}
\setlength{\ItemLength}{\widthof{\bfseries Step 9 -- }}
\begin{proof}
\begin{enumerate}[label={\bfseries Step \arabic* --}, leftmargin=0pt, labelwidth=-\ItemLength] 
\item \label{item_pf_lemma_sums_sigmab_1}
The first key point here is to use again the fact that the map $g_{b, j}$ is additive, in order to deduce that $\sigma_b(j, t)$ takes only two values (for fixed $j, b$ and variable $t$).

Indeed, if we pick any $x_0 \in \fieldJ$, then
\begin{equation*}
\begin{split}
\sigma_b(j, t)	
&\overset{		\eqref{eq_definition_sums_sigmab}	}{=}	
\sum_{  x  \in \fieldJ  }  \lambda_{ \fieldJ }		\!\left(  g_{b, j}(x)  + t		\right)
=
\sum_{  x' \in \fieldJ  }  \lambda_{ \fieldJ }	\!\left(		g_{b, j}(x' + x_0)  + t		\right)
\\&\;\;= 
\sum_{  x' \in \fieldJ  }  \lambda_{ \fieldJ }\left( g_{b, j}(x') \;+\; g_{b, j}(x_0)  + t \right)
= \sigma_b(j, t + g_{b, j}(x_0)).
\end{split}
\end{equation*}
In other words, $\sigma_b(j, t)$ only depends on the class of $t$ in the quotient additive group $\fieldJ \big/ \im(g_{b, j})$.
Moreover, notice that
\begin{equation*}
\begin{split}
\sigma_b(j, t)	
&=	\sum_{  x' \in \fieldJ  }  \lambda_{ \fieldJ }\big( g_{b, j}(-x')  + t \big)
= 		\sum_{  x' \in \fieldJ  }  \lambda_{ \fieldJ }\big( \!-\! g_{b, j}(x') + t \big)
\\&= 
\lambda_{ \fieldJ }(- 1)	\cdot \sigma_b(j, -t) 
= \sigma_b(j, -t),
\end{split}
\end{equation*}
where the last equality holds because $-1$ is a square in $\F_{3^2}$ and hence in $\field$.

Since $[ \fieldJ : \im(g_{b, j}) ] = |\ker(g_{b, j})| = 3$ (because $-b \in \F_q$ is a square in $\F_{q^2} \into \F_{q^{2j}}$), 
we deduce that $\sigma_b(j, t)$ only takes two values (for fixed $j, b$ and variable $t$). The first value
occurs when $t \in \im(g_{b, j})$ in which case $\sigma_b(j, t) = \sigma_b(j, 0)$. Let us denote by $\sigma^*$ the other value of $\sigma_b(j, t)$, which occurs when $t \not\in \im(g_{b, j})$. 
Observe that the value of $\sigma^*$ can be deduced from the sum
\[ \sum_{ 		\mathclap{t \in 		 \field} 	   }    \sigma_b(j, t)   =  
|\im(g_{b, j})|  \cdot    \sigma_b(j, 0) 			\;+\;			\left(3^{2 nj} - |\im(g_{b, j})|\right)  \cdot \sigma^*
=
3^{2 nj} \left(  \frac{1}{3} \sigma_b(j, 0)		\;+\;		\frac{2}{3} \sigma^*  \right) \]
because the left-hand side sum vanishes :
\[ 
\sum_{ 		\mathclap{t \in 	\field } 	   }    \sigma_b(j, t)
\;=
\sum_{ x \in  	\field  }
\sum_{t \in 		 \field }   		 \lambda_{ \field }( x^3 + b x + t ) = 0,
 \]
 since all the inner sums are $0$ (they are sums of a non-trivial multiplicative character over the whole group -- recall also that $\lambda_{ \field }(0) = 0$).
 Therefore $\sigma^* = -\frac{1}{2} \sigma_b(j, 0)$, so it is enough to determine the value of $\sigma_b(j, 0)$.

\medskip
\item 
Now we compute the sum
$\sigma_b(j, 0)	= \sum\limits_{x \in \fieldJ} 		\lambda_{	\fieldJ	}(x^3 + bx)$.

The most conceptual (and easiest, or shortest) proof relies on the fact that if ${\pi : Y \to X}$ is a surjective morphism between two smooth irreducible projective algebraic curves (or even varieties) defined over a finite field, then the numerator of the zeta function of $X$ divides the one of $Y$ in $\Z[T]$. This can be argued using the Tate modules of the jacobians of these curves, see for instance proposition 5 in \cite{Aubry_Perret_Zeta_divisibility}.

In our case, we have the morphism
\[
\pi : C_{n, b} \to \mathcal{E}_b  \qquad		(u, v) \mapsto \Big(	u, v^{ 	\tfrac{3^n + 1}{2}	 }	\Big)
\]
where $\mathcal{E}_b$ is the elliptic curve given by $y^2 = x^3 + bx$ over $\F_{3^n}$ (we defined the morphism on the affine charts, but it extends uniquely to a morphism between the smooth projective curves $C_{n, b} \to \mathcal{E}_b$). 
Being a non-constant morphism between irreducible curves, $\pi$ must be surjective.

The numerator of $Z(C_{n, b} / \F_{3^{2n}}, T)$ is $(1 + 3^n T)^{2 \cdot 3^n}$ by \cref{prop_zeta_function_superelliptic}.
Therefore, the numerator of $Z(\mathcal{E}_b /  \F_{3^{2n}}, T)$ is $(1 + 3^n T)^{2}$ (which implies that $\mathcal{E}^b$ is supersingular). 
Thus we deduce from standard arguments (see \cite{SilvermanAEC}, application V.1.3 and theorem V.2.3.1) that 
\begin{align*}
1 + 3^{2nj} + \sigma_b(j, 0)
\;=\;
|\mathcal{E}_b( \field )|
\;=\;
1 + 3^{2nj} - 2 (-3^n)^{j},
\end{align*}
which gives the claimed value for $\sigma_b(j, 0)$. Therefore, from \hyperref[item_pf_lemma_sums_sigmab_1]{step 1} we get the value $\sigma^* = (-3^n)^j$ and this finishes the proof.
\qedhere
\end{enumerate}
\end{proof}

\begin{rmq}
It is possible to give more concrete and elementary (but computationally longer) proofs of the identity $\sigma_b(j, 0) = -2 \cdot (-3^n)^j$ from \cref{lemma_sums_sigmab}, via quartic Jacobi sums.  

Moreover, when $n$ is odd, one can also give a direct proof of the step 2 above, because the change of variables $x \mapsto -x$ allows to determine the number of points of the elliptic curve $\mathcal{E}_b : y^2 = x^3+bx$ over $\F_{3^n}$ (because $-1$ is not a square in $\F_{3^n}$) 
and hence over any field extension thereof. 
\end{rmq}

We are now in position to prove our main result.

\begin{proof}[Proof of \cref{thm_L_function}]
By the identity just below \cref{eq_definition_sums_Sb}, we recall that\linebreak 
$\displaystyle \log L(E_{n, b} / \F_{q^2}(t), T) = \sum\limits_{j \geq 1} 	S_b(n, j) \frac{ T^j }{ j }.$

From \cref{eq_sum_Sb_as_Legendre,eq_definition_sums_sigmab}, one can write
\begin{equation*}
- S_b(n, j)		=	  	\sum_{		w \in \field		}		\sigma_b(j, w^{3^n+1})
\end{equation*}
(be careful of the minus sign).
Define the set
\[
\Gamma_b(n, j)	:=		\left\{		w \in \field		\tq		w^{3^n+1} \in \im(g_{b, j})		\right\},
\]
where $g_{b, j} : \field \to \field$ denotes the map $x \mapsto x^3+bx$ as in \cref{lemma_sums_sigmab}.

Notice that all the fibers of the map
$$C_{n, b}^{\rm aff}(  \field  )  			\longrightarrow  			\Gamma_b(n, j),  	\qquad
	(u, v) \longmapsto v$$
have size $3$ (they have the shape $\{ (u, v) ; (u \pm \beta, v) \}$, where $\beta \in \F_{3^{2n}} \into \field$ is an element such that $\beta^2 = -b$).
Thereby, we deduce from \cref{prop_zeta_function_superelliptic} that 
\begin{equation}\label{eq_Gamma_values}
|\Gamma_b(n, j)| = \dfrac{1}{3}( |C_{n, b}(  \field  )|  		-		1)
=
\dfrac{1}{3}( 3^{2nj} - 2 \cdot 3^n \cdot (-3^n)^j )
\end{equation}

Therefore, using \cref{lemma_sums_sigmab} and the above expression of $S_b(n, j)$, we get
\begin{align*}
- S_b(n, j)		&=
-2 \cdot (-3^n)^{j} \cdot |\Gamma_b(n, j)|			\quad + \quad			(-3^n)^{j} \cdot \big( 3^{2nj} - |\Gamma_b(n,j)| \big)
\\ & = (-3^n)^{j} \cdot 	\big( 3^{2nj} - 3  \cdot |\Gamma_b(n,j)|	\big)
\\& \overset{\eqref{eq_Gamma_values}}{=} 
	(-3^n)^{j} \cdot 2 \cdot 3^n \cdot (-3^n)^j
\\&= 2 \cdot 3^{n(1+2j)} = 2 q^{1+2j},
\end{align*}

Finally, we conclude that
\begin{align*}
\log L(E_{n, b} / \F_{q^2}(t), T) 
&= \sum\limits_{j \geq 1} 	S_b(n, j) \frac{ T^j }{ j } \\
&= -2q \sum\limits_{j \geq 1}  \frac{ (q^2 T)^j }{ j } \\
&= 2q \cdot \log(1 - q^2 T),
\end{align*}
which precisely means that
\[
L\left( E_{n, b} / \F_{q^2}(t)	,\;	 T \right)  	=   (1 - q^2 T)^{2 \cdot 3^n},
\]
as desired. This finishes the proof.
\end{proof}

\begin{rmq}\label{rmk_Shioda_algo_2}
%
%
%
We explain why the case of characteristic $3$ is very special. For an odd prime $p$, the elliptic surface (of Delsarte type in Shioda's terminology from \cite{Shioda_algorithm_Picard_nb}) associated to $E : y^2=x^3+x+t^m$ over $\F_p$ is birationally equivalent to a quotient of the Fermat surface $\mathcal{F}_d$ of degree $d$, where $d := \frac{4m}{ \mathrm{gcd}(2, m) }$. 
So one can follow the approach taken in \cite{Griffon_Ulmer} to express the $L$-function of $E$ in terms of Jacobi sums, like $j(\theta, \theta^2) = \sum_{x \in k} \theta(x) \theta^2(1-x)$ for some suitable multiplicative characters $\theta$ (of order dividing $d$) on finite extensions $k$ of $\F_p$.

If $p^e \equiv -1 \pmod d$ for some integer $e \geq 1$, then one can apply \cite[Proposition 8.1]{Ulmer_Large_Rank} to compute explicitly those Jacobi sums. However, in our case where $m = p^n + 1$, this condition is not fulfilled so in general this does not allow to compute $j(\theta, \theta^2)$ directly. But in characteristic $p=3$, we have (when $\theta^6$ is not trivial) $j(\theta, \theta^2) = \dfrac{ g(\theta) g(\theta^2) }{ g(\theta^3) } = g(\theta^2)$, where $\displaystyle g(\chi) := \sum_{x \in k} \chi(x) \exp\Big( \frac{2 \pi i}{p} \mathrm{tr}_{k / \F_p}(x) \Big)$ is the Gauss sum corresponding to a multiplicative character $\chi$ on $k$.  We can then apply Tate--Shafarevitch's lemma \cite[Lemma 8.3]{Ulmer_Large_Rank} to compute $g(\theta^2)$ explicitly in the case $m = 3^n + 1$.
%
\end{rmq}

{\hypersetup{linkcolor=black!75}
\section[Proof of corollary A]{Proof of  \texorpdfstring{\cref{cor_packing_density_MWL}}{corollary A}}
\label{sect_pf_cor_A}
}

We now turn to the proof of the corollary concerning the narrow Mordell--Weil lattice attached to the elliptic curves $E_{n, b}$ (see \cref{def_narrow_MWL}), and the lower bound on its sphere packing density (see \cref{def_center_density}).

Estimating the sphere packing density of a lattice $L$ requires three steps :
\begin{enumerate}
\item Determine the rank of $L$. In the case of the Mordell--Weil lattice of $E_{n, b}$, this is essentially done in \cref{thm_L_function}.

\item Get an upper bound on the covolume of $L$. In our case, this is achieved by using the so-called Birch--Swinnerton-Dyer formula which we discuss below.

\item Finally, get a lower bound on the minimal non-zero norm in $L$. In the context of the narrow Mordell--Weil lattices, we use a result of Shioda (see \cref{thm_Lower_bound_on_minimal_norm} below).
\end{enumerate}

\subsection{Birch--Swinnerton-Dyer conjecture and formula}

We briefly recall what the Birch--Swinnerton-Dyer (BSD) conjecture is, and what is known about it. Originally, it was stated for elliptic curves over $\Q$, but it was then generalized to abelian varieties over any global field. However, for the sake of simplicity, we will stick to the case of elliptic curves over function fields, as given in \cite[conjecture 2.10]{Gross_lectures_BSD}. 

Theorem 2.6 \emph{ibid.} states that the $L$-function $L(E/K, T)$ of any non-constant elliptic curve over a global function field is a polynomial in $T$ with integral coefficients. In the case of the curves $E_{n, b}$ defined above, \cref{thm_L_function} provides a proof of the fact $L(E/K, T) \in \Z[T]$. In particular, this allows us to speak of the order of vanishing of the $L$-function at any given value of $T$ in $\C$.
Before stating the conjecture, we introduce some (standard) notations :
\begin{df}\label{def_Reg_and_special_value}
Let $k$ be a finite field, and let $X$ be a smooth projective geometrically irreducible algebraic curve over $k$. Denote by $g_X$ its genus. Set $K = k(X)$ and let $E$ be an elliptic curve over $K$.
\begin{enumerate}
\item 
\nomenclature[1\(R\)]{$\Reg(E/K)$}{Regulator of an elliptic curve $E/K$}{nomencl:Regulator}
Given the Néron--Tate height $\hat{h} : E(K) \to \R_{\geq 0}$  as in \cref{eq_def_Neron_Tate_height}, we define the pairing
\[	\langle - , - \rangle	: E(K) \to \R,		\quad 	
(P, Q) \mapsto \frac{1}{2} \cdot \left( \hat{h}(P+Q) - \hat{h}(P) - \hat{h}(Q) 	\right).  \]
Then the \emph{regulator} of $E/K$ is the discriminant of this pairing, and we denote it by
$\Reg(E/K) := \det\Big(  (\langle P_i, P_j \rangle)_{ \scriptsize 1 \leq i, j \leq r } \Big)$, where $\{P_1, ..., P_r\}$ is any $\Z$-basis of the free abelian group $E(K) \big/ E(K)_{\tors}$ (we set $\Reg(E/K) = 1$ by convention if the rank is $r=0$).

\item 
\nomenclature[1\(L\)]{$L^*(E/K)$}{Special value of the $L$-function of $E/K$ at $s=1$}{nomencl:SpecialValueL}
We further set the \emph{special value} of the $L$-function of $E/K$ to be
\[	L^*(E/K)	
\; := \;
		\dfrac{1}{\rho !} L^{(\rho)}(E/K, T) \Big\vert_{T = |k|^{-1}}
\]
where $\rho = \rho(E/K) := \ord_{T = |k|^{-1}} L(E / K, T)$ denotes the \emph{analytic rank}.

\item 
The Tate--Shafarevitch group is defined as
\[
\Sha(E / K) := \ker\Big( 
H^1(G_K, E(K^{\sep}))		\overset{\mathrm{res}}{\longrightarrow}
\prod_{v \text{ places of } K}	
H^1(G_{K_v}, E(K_v^{\sep}))	  \Big).
\]
where $G_K := \Gal(K^{\sep} / K)$ denotes the absolute Galois group (and same for each $K_v$), and the map is induced by the restriction of cocycles from $G_K$ to $G_{K_v}$ (using embeddings $K^{\sep} \into K_v^{\sep}$).

\item 
\nomenclature[1\(H\)]{$H(E/K)$}{Height of an elliptic curve $E/K$, given by $H(E/K) = "|k"|^{ \deg(\Delta_{\min}) / 12 }$}{nomencl:HeightEllCurve}
Finally, we define the \emph{height} of $E/K$ as 
$	H(E/K) 		:=		 |k|^{ 		 \tfrac{	\discr{E/K}	}{ 12 } 		}.	$
\qedhere
\end{enumerate}
\end{df}

\begin{rmq}
Because the $L$-function is a rational function in $\Q(T)$, we also have $L^*(E/K) = \left. \dfrac{ L(E/K, T) }{ (1-|k|T)^\rho } \right|_{T = |k|^{-1}}$ and this is a non-zero rational number.

There is another normalization of the Néron--Tate height, which is $\hat{h'} := \log(|k|) \cdot \hat{h}$, as in \cite{Gross_lectures_BSD} (lecture 3, §2). 
In that case, for the \hyperref[item_BSD_formula]{BSD formula} to be true, one has to take the special value of the \emph{complex} $L$-function, namely the value $\mathcal{L}^*(E/K)$ such that 
$$\mathcal{L}(E/K,s) := L(E/K, |k|^{-s})	\sim  \mathcal{L}^*(E/K)  \cdot  (s-1)^{\rho},
\qquad\text{as $s \to 1$.}$$

The two normalizations are consistent. Indeed, on the one hand, if one defines
$\Reg'(E/K)$ as the discriminant with respect to the pairing associated to $\hat{h'}$ (as in \cref{def_Reg_and_special_value}), then one has $\Reg'(E/K) = \log(|k|)^r \, \Reg(E/K)$. 
On the other hand, since $1 - |k|^{1-s} \sim \log(|k|) ( s-1 )$ as $s \to 1$, we have $\mathcal{L}^*(E/K) = \log(|k|)^r  \, L^*(E/K)$.		

We make the choice of using $\hat{h}$ and not $\hat{h'}$ because then the narrow Mordell--Weil group $E(K)^0$ (\cref{def_narrow_MWL}) becomes an \emph{integral} lattice (see \cref{thm_Lower_bound_on_minimal_norm}).
\end{rmq}

\begin{conj}[Birch--Swinnerton-Dyer]\label{conj_BSD_statements}
Let $k$ be a finite field, and let $X$ be a smooth projective geometrically irreducible curve over $k$. Denote by $g_X$ its genus.
Let $E$ be an elliptic curve over the function field $K := k(X)$. 

Then the following statements hold:
\begin{enumerate}[a)]
\item \label{item_BSD_equality_ranks}
The rank of the finitely generated\footnote{This result of finite generation of $E(K)$ is known as Mordell--Weil theorem, further extended by Néron and Lang.}  abelian group $E(K)$ is equal to the order of vanishing of the $L$-function of $E/K$ at $T = |k|^{-1}$, i.e.,
\[		\rk_{\Z}\!\big( E(K)	\big)		=		\ord_{T = |k|^{-1}} L(E / K, T). 				\]

\item \label{item_BSD_formula}
The Tate--Shafarevitch group $\Sha(E/K)$ is finite and we have the following identity, called BSD formula (using notations from \cref{def_Reg_and_special_value,def_Tamagawa_nb_discr_conductor}):
\begin{equation}\label{eq_BSD_formula}
L^*(E/K)		\;=\;		
\dfrac{ |\Sha(E/K)| \cdot \Reg(E/K)  \cdot c(E/K) }
		{ |E(K)_{\tors}|^{2}   \cdot  |k|^{g_X - 1}  \cdot H(E/K)  }.
\end{equation}
\end{enumerate}
\end{conj}

\begin{thm}[Artin, Tate, Milne]\label{thm_BSD_true_for_isotrivial}
Let $E$ be an elliptic curve over the function field $K := k(X)$, where $k$ is a finite field, as in \cref{def_Reg_and_special_value}.

\begin{enumerate}
\item The statements \ref{item_BSD_equality_ranks} and \ref{item_BSD_formula} in \cref{conj_BSD_statements} are all equivalent.

\item Assume that $E$ is a potentially constant (= isotrivial) elliptic curve, i.e., there is a finite extension $K' / K$ such that the base change $E \times_K K'$ is isomorphic to $E' \times_k K'$ for some (constant) elliptic curve $E'$ defined over $k$.

Then all the statements of \cref{conj_BSD_statements} are true.
\qedhere
\end{enumerate}
\end{thm}
\begin{proof}
The first part is proved in \cite[Theorem 8.1]{Milne_conjecture_Artin_Tate}.
%
The second claim is stated in lecture 1, theorem 12.2 of \cite{Ulmer_elliptic_curves_over_FF}, and is proved in lecture 3, theorem~8.1,~\emph{ibid}.
\end{proof}

\begin{prop}\label{prop_E_isotrivial}
The elliptic curve $E_{n, b}$ over $K = \F_{3^{2n}}(t)$ from \cref{thm_L_function} is isotrivial. More precisely, it is a \emph{cubic twist} of the constant curve $E' : y'^2 = x'^3 + b x'$ over $\F_{3^n}$.

Moreover, the Mordell--Weil group $E_{n, b}(K)$ is torsion-free.
\end{prop}
\begin{proof}
The first statement is immediate from the change of variables $y = y', x = x' - u$ where $u \in \overline{\F_3(t)}$ satisfies $u^3 + bu = t^{3^n + 1}$ (this exactly defines the superelliptic curve from \cref{subsec_superelliptic_curve}). One can also see that the $j$-invariant of $E_{n, b}$ is $0$, so it must be an isotrivial elliptic curve.

We now explain why $E_{n, b}(K)$ is torsion-free. If we consider the cubic extension $K' := K(u)$ of $K$, with $u \in K$ as above, then we have an isomorphism $E'(K') \overset{\cong}{\to} E_{n, b}(K'), (x', y') \mapsto (x' - u, y')$ and $E'(K')_{\tors} = E'(\F_{3^{2n}})$ by \cite[Proposition 6.1, lecture 1]{Ulmer_elliptic_curves_over_FF}. Since $x' - u \not\in K$ whenever $x' \in \F_{3^{2n}}$, this proves that $E_{n, b}(K)$ has to be trivial.  

Alternatively, one can prove that $E_{n, b}(K)$ is torsion-free as follows: \cref{prop_reduction_types} implies that the product of the Tamagawa numbers is equal to $c(E_{n, b} / K) = \prod\limits_v c_v = 3$. In particular, this is a square-free integer. But \cite[proposition 6.31]{MW_Lattices_Book} states that $|E_{n, b}(K)_{\tors}|^2$ divides $\prod\limits_{v} c_v(E_{n, b} / K)$, so we deduce that $E_{n, b}(K)$ is torsion-free.
\end{proof}

\subsection{Lower bound on the minimal norm and on the packing density}

We start this subsection in a general framework : we let $E$ be an elliptic curve over a global function field $K=k(X)$ as in \cref{def_Reg_and_special_value}, that is, $X$ is a smooth geometrically irreducible projective curve over a finite field $k$.

One of the main features of the \emph{narrow} Mordell--Weil lattice $E(K)^0 \subset E(K)$ (\cref{def_narrow_MWL}) is that it is an \emph{even integral} lattice, and that we have an explicit lower bound on the minimal height among non-zero vectors.
\begin{thm}[Shioda]\label{thm_Lower_bound_on_minimal_norm}
Let $E$ be an elliptic curve over a global function field $K=k(X)$. Then for every $P \in E(K)^0 \minus \{0\}$ we have
\[ 		\hat h(P) 				\geq			 \frac{1}{6} 	\discr{E/K}.		\]
In particular, $E(K)^0$ is torsion-free. 
Moreover, $(E(K)^0, \hat{h})$ forms an {even integral} lattice.

Finally, the index $\left[  E(K) : E(K)^0  \right]$ divides the product $c(E/K) := \prod\limits_{v} c_v(E /K)$ of the Tamagawa numbers.
\end{thm}

\begin{proof}
For the lower bound on the minimal non-zero norm and the fact that the lattice $E(K)^0$ is even and integral, see theorem 6.44 in \cite{MW_Lattices_Book}, as well as theorem 5.47 and corollary 5.50, \emph{ibid}.

We now prove the result on the index $\left[  E(K) : E(K)^0  \right]$. 
Let $R \subset |X|$ be the set of bad places of $E$, where $|X|$ denotes the set of closed points of $X$. For each $v \in R$, let $G_v := \dfrac{   \tilde{\mathcal{E}_v}(\F_v)    }{   \tilde{\mathcal{E}_v}^0(\F_v)   }$ be the component group at $v$, where $\mathcal{E}_v$ denotes the Néron model of $E$ at $v$ and $\F_v$ is the residue field.

By definition and by \cite[corollary IV.9.2.(c)]{SilvermanAAEC}, $E(K)^0$ is the kernel of the map 
$$\theta : E(K) \longrightarrow \prod_{v \in R} G_v$$   
defined as follows: 
for each $v \in R$, there is a unique irreducible component $\Theta_{v, i(v, P)}$ of $\tilde{\mathcal{E}_v}$ that contains the image $\widetilde{P_v}$ of $P$ in $\tilde{\mathcal{E}_v}$. Then $P \longmapsto (\Theta_{v, i(v, P)})_{v \in R}$ induces the above map $\theta$.

The map $\theta$ is a group homomorphism (see lemma 6.4 in \cite{Elliptic_Surfaces_Shioda_Schutt}, or just notice that $E(K_v) \cong \mathcal{E}_v(\O_v) \to \tilde{ \mathcal{E}_v }(\F_v)$ is a morphism) and
therefore, we have an injective morphism  
\begin{equation*}\label{quot_MWL_by_narrow}
E(K) / E(K)^0   \longinto	 \prod_{v \in R} G_v,
\end{equation*}
which shows the desired divisibility.
\end{proof}

We can now give a lower bound on the sphere packing density of the narrow Mordell--Weil sublattice $E(K)^0 \subset E(K)$ (see \cref{def_center_density}).
\begin{prop}\label{prop_lower_bound_density_general_formula}
Let $E$ be an elliptic curve over a global function field $K = k(X)$, where $X$ is a smooth projective curve of genus $g_X$ over a finite field $k$.

Assume that the $L$-function of $E/K$ is of the form $L(E/K, T) = (1 - |k| T)^r$ where $r$ is the rank of $E(K) / E(K)_{\tors}$.

Then the center (sphere packing) density of the narrow Mordell--Weil lattice\linebreak $E(K)^0 \subset E(K)$ (see \cref{def_narrow_MWL,def_center_density}) is bounded below by 
\begin{equation*}
\delta\left(		E(K)^0		\right)
\;\geq\;
\dfrac{ 		\left(  \dfrac{\discr{E/K}}{24} \right)^{r / 2}  													}
		 {   	c(E/K)^{1/2} \cdot |E(K)_{\tors}|   \cdot  |k|^{g_X /2 - 1/2}  \cdot H(E/K)^{1/2} 		},
\end{equation*}
where we use the notations from \cref{def_Reg_and_special_value,conj_BSD_statements}.
\end{prop}
\begin{proof}
First of all, the hypothesis $L(E/K, T) = (1 - |k| T)^r$ implies that \hyperref[item_BSD_formula]{BSD formula} is true, by part 1 of \cref{thm_BSD_true_for_isotrivial} --- more precisely we used the implication \ref{item_BSD_equality_ranks} $\implies$ \ref{item_BSD_formula}.
This hypothesis also forces the special value of the $L$-function to be $L^*(E/K) = 1$.

Because the cardinality of the finite group $\Sha(E/K)$ is at least $1$, BSD formula allows us to get an upper bound on the discriminant of $E(K)$ :
\begin{equation}\label{eq_upper_bound_Reg}
 \Reg(E/K)
\leq
|E(K)_{\tors}|^{2}   \cdot  |k|^{g_X - 1}  \cdot H(E/K)  \cdot c(E/K)^{-1}
\end{equation}

From \cref{thm_Lower_bound_on_minimal_norm}, we have
\begin{equation*}
\lambda_1\left(		E(K)^0		\right)	
:=  
\min\Big\{ 		\hat{h}(P)^{1/2}		\tq		P \in L_n \minus \{0\} 		\Big\}
\;\geq\;
\left(			\dfrac{\discr{E/K}}{6} 		\right)^{1/2}.
\end{equation*}

Now the covolume of $E(K)^0$ is given by
\[
\covol\left(		E(K)^0		\right) 
= 
[E(K) : E(K)^0] \cdot \covol(E(K))
=
[E(K) : E(K)^0]  \cdot \Reg(E/K)^{1/2}.
\]

Using the last statement of \cref{thm_Lower_bound_on_minimal_norm}, together with \cref{eq_upper_bound_Reg}, we deduce
\begin{equation*}
\covol\left(		E(K)^0		\right) 
	\leq
c(E/K)^{1/2} \cdot |E(K)_{\tors}|   \cdot  |k|^{g_X/2 - 1/2}  \cdot H(E/K)^{1/2} 
\end{equation*}

Thereby, combining the above inequalities, we see that the center density of the lattice $L_n = E(K)^0$ is bounded below by 
\begin{equation*}
\delta\left(		E(K)^0		\right) 
\;\geq\;
\dfrac{ 		\left(  \dfrac{\discr{E/K}}{24} \right)^{r / 2}  													}
		 {   	c(E/K)^{1/2} \cdot |E(K)_{\tors}|   \cdot  |k|^{g_X/2 - 1/2}  \cdot H(E/K)^{1/2} 		},
\end{equation*}
where $r$ is the rank of lattice $E(K)^0$. Notice that the narrow Mordell--Weil\linebreak $E(K)^0 \subset E(K)$ is a \emph{full-rank} sublattice (this follows for instance from the last statement of \cref{thm_Lower_bound_on_minimal_norm} : its index in $E(K)$ is finite), so its rank is the same as the rank of $E(K)$.
\end{proof}

We can now conclude with the proof of our main corollary.
\begin{proof}[Proof of \cref{cor_packing_density_MWL}]
For ease of notation, in what follows, we write $K_n :=  \F_{3^{2n}}(t)$.

First of all, we notice that the rank of the lattice $L_n := E_{n, b}\left( K_n \right)^0$ is equal to $r = 2 \cdot 3^n$. Indeed, \cref{thm_BSD_true_for_isotrivial} and \cref{prop_E_isotrivial} imply that the BSD conjecture (item \ref{item_BSD_equality_ranks}) is fulfilled. In particular, the algebraic rank of $E_{n, b}$ over $K_n$ agrees with the analytic rank, which equals $2 \cdot 3^n$ by \cref{thm_L_function}. 

This very theorem also allows us to apply the above \cref{prop_lower_bound_density_general_formula}.
%
%
Thereby, the values from \cref{prop_reduction_types} and the last statement of \cref{prop_E_isotrivial} (namely the fact $|E_{n,b}(K_n)_{\tors}| = 1$) yield
\begin{equation*}
\delta(L_n)
\;\geq\;
\dfrac{ 		\Big(  (3^{n-1} + 1)/4	 \Big)^{3^n}  		}
		 {   	3^{1/2} \cdot  3^{n/2 \,\cdot\, (3^{n-1} - 1)}		},
\end{equation*}
which is exactly the lower bound stated in \cref{cor_packing_density_MWL}. This concludes the proof.
\end{proof}

\subsection{Discussion of the sharpness of the lower bound on the {packing} density}\label{subsec_sharpness}

In this paragraph, we shorty study sufficient conditions under which the inequality in \cref{cor_packing_density_MWL} is actually an equality.
In fact, this lower bound is sharp if and only the following conditions are all satisfied :
\begin{itemize}
\item 
The index $[E(K) : E(K)^0]$ is equal to $c(E/K)$ (instead of just dividing it, as in \cref{thm_Lower_bound_on_minimal_norm}).

\item 
The lower bound on the minimal norm form \cref{thm_Lower_bound_on_minimal_norm} is achieved, that is there is a point $P \in E(K)^0$ such that $\hat{h}(P) = \dfrac{1}{6} \discr{E/K}$, which is equal to $3^{n-1} + 1$ when $E = E_{n, b}$ according to \cref{prop_reduction_types}.

\item 
The Tate--Shafarevitch group $\Sha(E/K)$ is trivial.
\end{itemize}

\bigskip

As for the index $[E_{n, b}(K) : E_{n, b}(K)^0]$, where $K := \F_{3^{2n}}(t)$, we can prove easily that it is in fact equal to $c(E_{n, b}/K) = 3$. First, we know from the last statement of \cref{thm_Lower_bound_on_minimal_norm} that $[E_{n, b}(K) : E_{n, b}(K)^0]$ must divide $c(E_{n, b} / K) = 3$, so it is either 1 or 3. 
We prove that the index cannot be equal to 1 by noticing that the point 
$$Q_n := \left(		0, t^{ (3^n+1)/2 }		\right) \in E_{n, b}( \F_3(t) ) \into E_{n, b}(K)$$ does not belong to $E_{n, b}(K)^0$.

Indeed, if we set $\mu = \lceil (3^n+1)/6 \rceil$, then the point $Q_n$ gets mapped to the point
$(Q_n)_{\infty} := \big(    0, t^{(3^n + 1)/2 - 3 \mu}    \big)$ on the minimal integral Weierstrass model 
$(E_{n, b})_v : y^2 = x^3 + bx t^{-4 \mu} + t^{3^n + 1 -6 \mu}$ of $E_{n, b}$ at $v := \infty$ (via the map $(x,y) \mapsto (x t^{-2 \mu}, y t^{-3 \mu})$), as in \cref{prop_reduction_types}. Then $(Q_n)_{\infty}$ modulo $t^{-1}$ is the singular point $(\overline{0}, \overline{0})$ of $\overline{(E_{n, b})_v}$. 
Therefore, $Q_n \not \in E_{n, b}(K)^0$, as claimed.

\bigskip

Let us say a few words on the lower bound $\hat{h}(P) \geq  \dfrac{1}{6} \discr{E_{n, b}/K} = 3^{n-1} + 1$ for $P \in E_{n, b}(K)^0 \minus \{0\}$.
We do not know whether it is a sharp bound in general, but for $n \in \{1, 2, 3\}$ we can exhibit points that achieve this bound. We first list those points explicitly, and then briefly explain how to compute their Néron--Tate height.
\begin{itemize}[label={---}, leftmargin=*]
\item When $n=1$ and $b=1$, the point 
$$P_1 = (t^2, -t^3 + t) 		\;\in\;		 E_{1, 1}(\F_3(t)) \longinto{} E_{1,1}(K)$$ 
has Néron--Tate height $2$, i.e., $\hat{h}(P_1) = 3^0 + 1$. Notice that $P_1$ lies in the narrow Mordell--Weil sublattice, because at $v = \infty$, the point $P_1$ gets mapped to
$(1, -1+t^{-2})$ on the minimal integral Weierstrass model at $\infty$, so it reduces to the smooth point $(\overline{1}, \overline{-1})$ modulo $t^{-1}$.

\item If $n=2$, let us write $\F_{3^2} \cong \F_3[X] / (X^2 - X - 1)$ and let $z$ be the class of $X$ in $\F_{3^2}$. One can take $b := z$ since $z^{(3^n - 1)/2} = z^4 = -1$. The point 
$$P_2 := \left(  t^4 + (z+1) t^2 - 1		\; , \;		 -t^6 + t^4 - t^2 - z + 1 \right)
 \in E_{2, b}(\F_{3^2}(t)) \into E_{2, b}(\F_{3^{2n}}(t))$$
has height 4.
Again, $P_2$ lies in the narrow Mordell--Weil sublattice : its reduction modulo $t^{-1}$ is $(\overline{1}, \overline{-1})$ as for the $n=1$ case.

\item If $n=3$ and $b=1$, then
$$P_3 = \left(
t^{10} + t^{8} + t^2 \; , \; 
-t^{15} + t^{13} - t^{11} - t^7 - t^5 + t
 \right) \in E_{3, 1}(K)$$
has height 10. Moreover, as before, $P_3$ lies in the narrow Mordell--Weil sublattice.
\end{itemize}

Using theorem 6.24 of \cite{MW_Lattices_Book}, one can show that $\hat{h}(P_n) = 3^{n-1} + 1$ for $n \leq 3$ by checking that the intersection product $(P_n) \cdot (O)$ vanishes, that is, the sections $(P_n)$ and $(O)$ -- from $\Pl^1$ to the elliptic surface associated to $E_{n,b}$ -- do not intersect (we use the notations from Proposition 5.4 and Notation 5.5 in \cite{MW_Lattices_Book}).

One can argue as in the proof of Proposition 5.1 of \cite{Shioda_MW_and_sphere_packings} (even though the exact statement from there does not directly apply in characteristic 3): 
both coordinates of $P_n$ are polynomials in $t$, so have no pole on $\A^1$, and hence $(P_n)$ and $(O)$ do not intersect at any point of $\A^1$. At $v = \infty \in \Pl^1$, we let $\mu = \lceil (3^n+1)/6 \rceil$ and observe that under the map $(x(t), y(t)) \mapsto (x(t) t^{-2 \mu}, x(t) t^{-3 \mu})$, the points $P_n$ get mapped to points $(P_n)_{\infty}$ on the minimal integral Weierstrass model of $E_{n,b}$ at $v=\infty$ such that both coordinates have non-zero constant term. Hence, we see that both coordinates have no pole at $v = \infty$, and we conclude that $(P_n)$ and $(O)$ never intersect.

\bigskip

Finally, for the order of the Tate--Shafarevitch group, we can just point out that it is a 3-group, i.e., it is equal to its 3-primary part $\Sha(E_{n, b} / K) = \Sha(E_{n, b} / K)[3^{\infty}]$, where $K = \F_{3^{2n}}(t)$. This follows from \hyperref[item_BSD_formula]{BSD formula} : because $L^*(E_{n, b} / K) = 1$, $|E_{n, b}(K)_{\tors}| = 1$ and $c(E_{n, b}/K) = 3$, we have
$$|\Sha(E_{n, b}/K)| \cdot \Reg(E_{n, b}/K) 		=		
\dfrac{1}{3} \cdot		(3^{2n})^{-1 +   \tfrac{1}{12} \discr{E_{n, b} / K}}  .$$
But we have seen above that $[E_{n, b}(K) : E_{n, b}(K)^0] = 3$, and we know from \cref{thm_Lower_bound_on_minimal_norm} that $E_{n, b}(K)^0$ is an integral lattice, so it follows that $ \Reg(E_{n, b}/K) \in \frac{1}{3^2}  \Z$.

Computations on MAGMA \cite{MAGMA} seem to indicate that $\Sha(E_{n, b} / K)$ is trivial when $n=1$, but in analogy with \cite[proposition 4.3, corollary 4.6]{Shioda_MW_and_sphere_packings}, it is possible that it is non-trivial for $n$ large enough. 
\begin{rmq}
In fact, when $n=1$, i.e., when the rank is $r = 2 \cdot 3^1 = 6$, it is known that the $E_6$ lattice provides the best \emph{lattice} sphere packing in 6 dimensions \cite{Blichfeldt_optimality_E6_7_8}, and since the lower bound on the density of the lattice $E_{1, 1}(\F_{3^2}(t))^0$ agrees 
with the density of $E_6$, the lower bound from \cref{cor_packing_density_MWL} must be sharp when $n=1$, in particular $\Sha(E_{1, 1} / K)$ is trivial. 
\end{rmq}


\begin{rmq}
\begin{enumerate}[leftmargin=3ex]
\item 
We mention here that when $n \to + \infty$, we have the asymptotic lower bound
$\log_2( \delta(L_n) ) \geq 3^n \cdot n \cdot \log_2(3)	-	\frac{n \cdot 3^{n-1}}{2} \log_2(3) + o(n \cdot 3^n)$ from \cref{cor_packing_density_MWL}.
Because the rank of $L_n$ is $r = 2 \cdot 3^n$, this reads
$$\log_2( \delta(L_n) ) \geq \left( \frac{1}{2} - \frac{1}{12} \right) r \log_2(r)  +  o(r \log_2(r)),$$ which implies 
\begin{align}\label{eq_lower_asymptotic_bound_packing_density_MWL}
D(L_n) \geq 2^{ - \tfrac{1}{12} r \log_2(r) \cdot (1+o(1))} = r^{-r/ 12  \cdot  (1+o(1))},
\end{align}
where $D(L_n) \in [0,1]$ is the packing density as defined in \cref{eq_def_packing_density}. Although this is far from attaining Minkowski--Hlawka lower bound $\geq 2^{-r}$, we get the same asymptotic density as in \cite[theorem 1]{ElkiesMW1} and \cite[equation (1.12)]{Shioda_MW_and_sphere_packings}.

\item 
We point out some key properties that are shared by the family of elliptic curves studied here and the ones from \cite{ElkiesMW1, Shioda_MW_and_sphere_packings}. 
Namely, all these three families $(E_i / \F_{q_i}(t))_{i \geq 1}$ of elliptic curves (ordered by increasing conductor) are such that:
\begin{itemize}
\item The Szpiro ratio $\sigma_i := \dfrac{  \discr{E_i}  }{ f(E_i) } \sim 1$ is asymptotic to $1$, as $i \to \infty$.

\item Brumer's bound \cite[proposition 6.9]{Brumer_average_rank} is asymptotically sharp: as $i \to \infty$ we have
\[  \rk(E_i / \F_{q_i}(t))    \;\sim\;			\dfrac{ f(E_i) \log(q_i) }{ 2 \log(f(E_i)) }	 \]
\end{itemize}

In fact, using the upper bound on the Brauer--Siegel ratio stated and proved in
\cite[Theorem 1.10]{Hindry_Pacheco}, one can show that the narrow Mordell--Weil lattices $(L_i)_{i \geq 1}$ attached to any family of non-constant elliptic curves satisfying the BSD conjecture and the two properties above, will satisfy the asymptotic lower bound \eqref{eq_lower_asymptotic_bound_packing_density_MWL}, as the conductor goes to infinity.

%
%

\item 
The densities of the narrow and the full Mordell--Weil lattices compare as follows. Let $Q_n := \left(		0, t^{ (3^n+1)/2 }		\right)$ be as above.
Using theorem 6.24 and Table 6.1 (p. 127) of \cite{MW_Lattices_Book} and the fact that the reduction of $E_{n,b}$ at $v=\infty$ has type IV (\cref{prop_reduction_types}), one can show that $\hat{h}(Q_n) =  3^{n-1} + 1 - \frac{2}{3}$, using an argument similar as the one for the computations of $\hat{h}(P_n)$ above. 
Then
$$
\delta\big( E_{n,b}(\F_{3^{2n}}(t) \big) 	\leq	
\dfrac{  \big( \hat{h}(Q_n)^{1/2}  /  2 \big)^{2 \cdot 3^n}    \cdot   [E_{n,b}(\F_{3^{2n}}(t)) : L_n]  }{  \covol(L_n)  }
$$

Thus we get, because $ [E_{n,b}(\F_{3^{2n}}(t)) : L_n] = 3 $ as mentioned previously, 
\begin{align*}
\dfrac{		\delta\big( E_{n,b}(\F_{3^{2n}}(t) \big)	  }{ 		\delta(L_n)		 }
\leq
3 \cdot \left( \dfrac{\hat{h}(Q_n)}{ \lambda_1(L_n)^2 } \right)^{3^n}
\leq 
3 \cdot  \left( \dfrac{3^{n-1} + 1 - 2/3}{3^{n-1} + 1} \right)^{3^n}
=
3 \cdot  \left( 1 - \dfrac{2}{3^{n} + 3} \right)^{3^n}
\end{align*}
Thus the narrow Mordell--Weil lattice $L_n$ is always denser than the full Mordell--Weil lattice, and the ratio of the densities tends to $3 e^{-2} \simeq 0.406 $ as $n \to + \infty$. 
\qedhere
\end{enumerate}
\end{rmq}

\subsubsection*{Acknowledgments}

{\hypersetup{urlcolor=black}
I would like to thank my advisor, Prof. Maryna Viazovska, for her support and for having suggested to study this topic. I also thank Vlad Serban, Matthew de Courcy-Ireland and the anonymous referee who gave me helpful comments on an earlier version of this paper.
 
This work was funded by the Swiss National Science Foundation (SNSF), Project funding (Div. I-III), "Optimal configurations in multidimensional spaces", 
\href{http://p3.snf.ch/project-184927}{\no 184927}.
}


\footnotesize


{
\input{"On_the_Mordell-Weil_lattice_of_y2_=_x3_+_bx_+_t_3n+1_in_characteristic_3__ArXiV".nls}
}

{
\usefont{T1}{Roboto-LF}{l}{n}
\pdfbookmark{References}{bibliography}
\input{"On_the_Mordell-Weil_lattice_of_y2_=_x3_+_bx_+_t_3n+1_in_characteristic_3__ArXiV_bibliography".bbl}
}

\vfill

Gauthier~\textsc{Leterrier}, {École Polytechnique Fédérale de Lausanne (EPFL), MA B3 424, Station 8, 1015 Lausanne, Switzerland}

\textit{E-mail address} : 
\texttt{\scriptsize gauthier.leterrier \textit{at} epfl \textit{dot} ch} 
\; or \; 
\texttt{\scriptsize gauthier.leterrier \textit{at} gmail \textit{dot} com}
\end{document}